\documentclass[12pt]{article}

\usepackage{amsmath, amsthm, amssymb}
\usepackage{geometry}
\usepackage{hyperref}
\usepackage{microtype}
\usepackage{tikz}
\usepackage{enumitem}
\usepackage{times}
\newtheorem{theorem}{Theorem}[section]
\newtheorem{lemma}[theorem]{Lemma}
\newtheorem{corollary}[theorem]{Corollary}

\theoremstyle{definition}
\newtheorem{definition}[theorem]{Definition}

\newtheorem{example}[theorem]{Example}

\DeclareMathOperator{\tr}{tr}
\DeclareMathOperator{\spec}{spec}
\newcommand{\TT}{\mathbb{T}}
\newcommand{\CC}{\mathbb{C}}
\newcommand{\vp}{\varphi}

\title{On Spectra of $\mathbb{T}$-Gain Digraphs}

\begin{document}

\author{Shivani Tushar Parab\footnote{shivaniparab004@gmail.com}, Tarkeshwar Singh\footnote{tksingh@goa.bits-pilani.ac.in} and Mukesh Kumar Nagar \footnote{mukesh.kr.nagar@gmail.com}\\
Department of Mathematics,\\
BITS Pilani K K Birla Goa Campus,
Goa, India.}
\date{}
\maketitle

\begin{abstract}
A \(\mathbb{T}\)-gain digraph is a directed graph with complex unit gains on its arcs, allowing for no restrictions on oppositely directed arcs. This framework unifies various graph types, such as signed graphs, mixed graphs, complex unit gain graphs, digraphs, and signed digraphs. The gain adjacency, Laplacian, and signless Laplacian matrices are generally non-Hermitian with complex spectra. We develop the spectral theory of these matrices, extending classical results from signed digraphs and complex unit-gain graphs. For the gain adjacency matrix, we establish determinant and characteristic polynomial formulas, characterize cycle balance through switching equivalence and cospectrality with the underlying digraph. We further bound the spectral radius in terms of the underlying digraph and its maximum out-degree, with equality characterized by $\mu$-balance. As a consequence, we determine the spectra of $\mathbb{T}$-gain unicyclic digraphs. 
Moreover, we 
%also obtain determinant and characteristic polynomial formulas,
characterize cycle balance and antibalance for the Laplacian and signless Laplacian matrices, respectively, through the presence of a zero eigenvalue.

\end{abstract}



\section{Introduction}
Graphs are among the most fundamental structures in discrete mathematics, providing a natural model for pairwise relationships between objects. A graph $\Gamma = (V, E)$ consists of a finite set of vertices $V$ together with a set of edges $E$, where each edge is an unordered pair of distinct vertices; two vertices are said to be \emph{adjacent} if they are joined by an edge.

Spectral graph theory studies graphs using the eigenvalues and eigenvectors of matrices associated with them, such as the adjacency matrix $A(\Gamma)$ and the Laplacian matrix $L(\Gamma)$ for a graph $\Gamma$. Many structural properties of a graph are reflected in the spectra of these matrices. For example, the entries of $A(\Gamma)^k$ count walks of length $k$, while $\operatorname{tr}(A(\Gamma)^k)$ gives the number of closed walks of length $k$. The spectrum also provides information about properties such as connectivity, bipartiteness and regularity, etc. Besides its importance in graph theory, spectral graph theory has found applications in chemistry, physics and computer science. For basic terminology
and standard results, we refer to \cite{Bapat}.

Many real-world networks, such as communication, transportation and biological networks, are naturally modelled by directed graphs, where the direction of an edge represents the direction of interaction or flow. Spectral graph theory for digraphs studies such networks through the eigenvalues of matrices associated with them. A \emph{digraph} is an ordered pair $D=(V,\mathcal{A})$, where $V = \{v_1,v_2,\dots,v_n\}$ is a finite vertex set and $\mathcal{A}$ is a set of ordered pairs of distinct vertices, called \emph{arcs}. For a vertex $v\in V$, the out-degree $\deg^{+}(v)$ is the number of arcs leaving $v$, and the in-degree $\deg^{-}(v)$ is the number of arcs entering $v$ and we set $\Delta^{+} = \max_i \deg^{+}(v_i)$, $\Delta^{-} = \max_i \deg^{-}(v_i)$. A \emph{directed walk} of length $k$ is a sequence $v_{i_0} \to v_{i_1} \to \cdots \to v_{i_k}$ of arcs, and is \emph{closed} if $v_{i_k} = v_{i_0}$. A \emph{directed path} is a directed walk in which all vertices are distinct. A \emph{directed cycle} is a closed directed path. A cycle of length $2$ is a \emph{digon}. The digraph $D$ is \emph{strongly connected} if every ordered pair of vertices is joined by a directed path. Throughout this paper, all digraphs are finite and loopless, and for every ordered pair of distinct vertices $(v_i,v_j)$, there is at most one arc from $v_i$ to $v_j$. 

For an $n \times n$ matrix $M(\mathbb{C})$ over complex numbers $\mathbb{C}$, we denote $M^{\top}$ for the transpose, $M^{*} = \overline{M}^{\top}$ for the conjugate transpose, $\spec(M)$ for its spectrum (the multiset of eigenvalues), $p_M(\lambda) = \det(\lambda I - M)$ its characteristic polynomial, and $\rho(M) = \max\{\,|\lambda| : \lambda \in \spec(M)\,\}$ its spectral radius. The \emph{adjacency matrix} of a digraph $D=(V,\mathcal{A})$ is the matrix $A(D)$  in which $(i,j)$th entry  $A(D)_{ij} = 1$ if $(v_i,v_j) \in A$ and $0$ otherwise. In general, the adjacency matrix $A(D)$ is not symmetric, and therefore its eigenvalues need not be real. Consequently, the spectrum of a digraph is, in general, a multiset of complex numbers. If every arc $(v_i,v_j)$ is accompanied by its reverse arc $(v_j,v_i)$, then $A(D)$ is symmetric and coincides with the adjacency matrix of the corresponding underlying undirected graph. We denote $\rho(D)$ as the spectral radius of $A(D)$. Thus the spectral theory of digraphs extends the classical spectral theory of graphs.

A \emph{gain graph}, as introduced by Zaslavsky~\cite{Zaslavsky, Zaslavsky1998}, is built on an undirected graph $\Gamma=(V,E)$, where $V$ and $E$ denote the vertex and edge sets of $\Gamma$. To assign gains,  one first fixes an \emph{orientation} of each edge, that is, for every edge $\{u,v\} \in E$ one chooses an ordered pair, say $(u,v)$, as its reference direction. A gain function $\varphi \colon E \to G$ then assigns to each oriented edge an element of a group $G$ with identity $e$, subject to the rule that reversing the orientation inverts the gain, $\varphi(v,u) = \varphi(u,v)^{-1}$ which is the inverse of group element $\varphi(u,v)$. Reading the gains consistently along the chosen orientations, a cycle $C = e_1, e_2, \ldots, e_k$ is \emph{balanced} precisely when its gain is the identity, i.e., $\varphi(C) \;=\; \varphi(e_1)\varphi(e_2)\cdots\varphi(e_k) \;=\; e .$
The gain group of interest is the circle group $\mathbb{T} = \{\, z \in \mathbb{C} : |z| = 1 \,\}$, a multiplicative abelian group in which $z^{-1} = \overline{z}$. Its cyclic subgroup of $k$th roots of unity is denoted $U_k$; in particular $\{\pm 1\} = U_2$ and $\{1,-1,\mathrm{i},-\mathrm{i}\} = U_4$. Any diagonal matrix $S = \operatorname{diag}(s_1,s_2,\dots,s_n)$ with $s_i \in \mathbb{T}$ is unitary, with $S^{-1} = S^{*}$.

The gain-graph framework has proved a natural home for several matrix constructions at once. The circle group $\mathbb{T}$ is large enough to absorb earlier ones: it contains the signed graphs as introduced by Harary~\cite{Harary}, where $G=\{\pm 1\} \subset \mathbb{T}$, and restricting the gains to the fourth roots of unity $\{\,1,-1,\mathrm{i},\mathrm{-i}\,\} \subset \mathbb{T}$ recovers the Hermitian adjacency matrix of a mixed graph, proposed independently by Liu and Li~\cite{LiuLi} and by Guo and Mohar~\cite{GuoMohar}, which has since produced a large literature on the interplay of orientation, gain, and spectrum. The spectral theory of complex unit-gain graphs was systematically developed by Reff~\cite{Reff}, who established a switching-invariant notion of balance, along with spectral-radius bounds in terms of the underlying graph.

 Acharya~\cite{Acharya} gave a spectral criterion for balance in networks: a network is cycle balanced exactly when its adjacency matrix is cospectral with its underlying nonnegative counterpart, which also yields a balance criterion for signed digraphs. Building on this, Pirzada and Bhat~\cite{PirzadaBhat} extended graph energy to signed digraphs $S$, whose eigenvalues $z_1,z_2,\dots,z_n$ are in general complex, setting $\mathcal{E}(S) = \sum_{j=1}^n |\operatorname{Re} z_j|$; they computed the energies of signed directed cycles and built infinite families of equienergetic signed digraphs. Nearby work treats the spectra and energy of bipartite signed digraphs~\cite{BhatPirzadaBipartite} and equienergetic signed graphs~\cite{BhatPirzadaEquienergetic} (the notion of equienergetic graphs goes back to Ramane et al.~\cite{Ramane}), finer variants such as the iota energy $\sum_j |\operatorname{Im} z_j|$~\cite{FarooqKhanChand}, complex-adjacency and energy studies of ordinary digraphs~\cite{KhanFarooqRada}, and the spectral characterisation of signed directed cycles~\cite{WissingVanDam}. In much of this directed work the arc labels stay in $\{\pm 1\}$, and one handles $\operatorname{Re} z_j$ and $\operatorname{Im} z_j$ separately.

All of these constructions are specializations of a single object, and it is this common generalization that the present paper takes up: the gain adjacency matrix of a $\mathbb{T}$-gain digraph. A $\mathbb{T}$-gain digraph also called as complex unit gain digraph generalizes the earlier objects in two separate ways. The first enlarges the gain group,
  $\{1\}\ \subset\ \{\pm1\}=U_2\ \subset\ U_4\ \subset\ \mathbb{T}$,
carrying an object from unlabelled, to signed, to the fourth-root labels underlying the Hermitian adjacency matrix, and finally to arbitrary complex unit gains. The second frees the orientation: rather than tying the two arcs of a digon together by the Hermitian relation $\varphi(v_j,v_i)=\overline{\varphi(v_i,v_j)}$, we impose no relation between opposite arcs at all. Putting back either restriction recovers a classical theory. Requiring all gains equal to $1$ returns the underlying digraph $D$; restricting the gains to $\{\pm1\}$ yields a signed digraph; and imposing the
Hermitian relation on the arcs turns a $\mathbb{T}$-gain digraph into a complex unit gain graph, which specializes further to the Hermitian adjacency matrix of a mixed graph when the gains lie in $U_4$, and to a signed graph when they lie in $\{\pm1\}$. Thus signed graphs, mixed graphs, complex unit gain graphs, digraphs, and signed digraphs all sit inside the $\mathbb{T}$-gain digraph as the cases in which one, or both, of these axes is held fixed; releasing both simultaneously is precisely what produces the genuinely complex, non-Hermitian spectral theory developed below.

The paper is organized as follows. Section \ref{sec 2} introduces T-gain digraphs, their adjacency matrices, switching, and cycle balance. It also develops determinant and characteristic polynomial formulae, which are used to show that a gain digraph has the same spectrum as its underlying digraph exactly when every directed cycle has gain 1; for strongly connected digraphs, this is equivalent to being switching equivalent to the all-ones gain. Section \ref{sec 3} studies the bounds on spectral radius , proving that gains cannot increase the spectral radius, characterizing the equality case, and deriving upper and lower bounds. Section \ref{sec4} determines the spectra of directed cycles and unicyclic gain digraphs, showing that the spectrum depends only on the number of vertices, the cycle length, and its gain. Section \ref{sec5} extends these results to the Laplacian and signless Laplacian matrices, giving determinant expansions, characterizing when 0 is an eigenvalue, and showing that all eigenvalues lie in the closed right half-plane. Finally, Section \ref{sec6} concludes the paper by discussing the research gaps and outlining directions for future work.

\section{$\mathbb{T}$-Gain Digraphs, Switching, and Balance} \label{sec 2}

We now introduce the central object of the paper. Let $G$ be a group with identity $e$. A \emph{$G$-gain digraph} is a pair $\Phi = (D,\varphi)$ consisting of a digraph $D = (V,\mathcal{A})$ and a \emph{gain function} $\varphi\colon \mathcal{A} \to G$ assigning a group element to each arc. Our interest is the case $G = \mathbb{T}$, and we call such a $\Phi$ a \emph{$\mathbb{T}$-gain digraph}; its spectral data are carried by the following matrix.

\begin{definition}\label{def:gainmatrix}
The \emph{gain adjacency matrix} $A(\Phi) \in M_n(\mathbb{C})$ of a
$\mathbb{T}$-gain digraph $\Phi = (D,\varphi)$ is
\[
  A(\Phi)_{ij} =
  \begin{cases}
    \varphi(v_i,v_j), & (v_i,v_j) \in \mathcal{A},\\
    0, & \text{otherwise.}
  \end{cases}
\]
\end{definition}

When $\varphi \equiv 1$, i.e. every arc is mapped to $1$ this reduces to the ordinary adjacency matrix, $A(\Phi) = A(D)$. Since $\varphi$ takes values in $\mathbb{T}$, every nonzero entry of $A(\Phi)$ has modulus one. We impose no relation between $\varphi(u,v)$ and $\varphi(v,u)$ when both arcs are present, so $A(\Phi)$ is in general neither symmetric nor Hermitian, and its spectrum may be complex. This matrix generalizes the adjacency matrix of the signed digraphs of
Acharya~\cite{Acharya}, which is the case where the gains are restricted to $\{\pm1\}$. We call $\det(\lambda I - A(\Phi))$ the characteristic polynomial of $\Phi$, and the eigenvalues of $A(\Phi)$ the eigenvalues of $\Phi$.

We call $\Phi$ \emph{symmetric} if reversing any arc gives another arc carrying the conjugate gain: whenever $(u,v)\in \mathcal{A}$, we also have $(v,u)\in \mathcal{A}$ with $\varphi(v,u)=\overline{\varphi(u,v)}=\varphi(u,v)^{-1}$. For a signed digraph this simply says that the two arcs of a digon carry the same sign. In general it forces $A(\Phi)$ to be Hermitian, so its spectrum is real. Symmetric $\mathbb{T}$-gain digraphs are exactly the complex unit gain graphs seen from the directed side: given a complex unit gain graph $\Gamma$, form $\overleftrightarrow{\Gamma}$ on the same vertex set by replacing each gain edge $\{u,v\}$, oriented so that $\varphi(u,v)$ labels one direction, with the pair of arcs $(u,v)$ and $(v,u)$ carrying gains $\varphi(u,v)$ and $\overline{\varphi(u,v)}$. The map $\Gamma\rightsquigarrow \overleftrightarrow{\Gamma}$ is a one-to-one correspondence, and through it a complex unit gain graph can be regarded as a symmetric $\mathbb{T}$-gain digraph.

\begin{example}\label{ex:running}
Let $D$ be a digraph with vertex set $V = \{v_1,v_2,v_3\}$ and arcs
$v_1 \to v_2,\ v_2 \to v_3,\ v_3 \to v_1,\ v_2 \to v_1$ as shown in Figure~\ref{fig:running}. Consider
$\omega = e^{2\pi \mathrm{i}/3}$ for a primitive cube root of unity, and set
\(
  \varphi(v_1,v_2) = \mathrm{i},\
  \varphi(v_2,v_3) = \omega,\
  \varphi(v_3,v_1) = 1,\
  \varphi(v_2,v_1) = \omega^{2}
\).
 Then $A(\Phi)$ is not Hermitian where,
\[
  A(\Phi) =
  \begin{pmatrix}
    0 & \mathrm{i} & 0 \\
    \omega^{2} & 0 & \omega \\
    1 & 0 & 0
  \end{pmatrix}.
\]

\begin{figure}[ht]
\centering
\begin{tikzpicture}[>=stealth, node distance=2.6cm,
  every node/.style={circle, draw, minimum size=6mm, inner sep=1pt}]
  \node (v1) {$v_1$};
  \node (v2) [right of=v1] {$v_2$};
  \node (v3) [below right of=v1] {$v_3$};
  \draw[->] (v1) to[bend left=15] node[draw=none, above] {$\mathrm{i}$} (v2);
  \draw[->] (v2) to[bend left=15] node[draw=none, below] {$\omega^{2}$} (v1);
  \draw[->] (v2) to node[draw=none, right] {$\omega$} (v3);
  \draw[->] (v3) to node[draw=none, left] {$1$} (v1);
\end{tikzpicture}
\caption{Example of $\mathbb{T}$-gain digraph $\Phi$.}
\label{fig:running}
\end{figure}
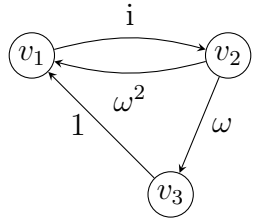
\end{example}

An \emph{elementary subdigraph} of $D$ is a subdigraph in which every vertex has in-degree and out-degree exactly $1$; equivalently, it is a vertex-disjoint union of directed cycles. For $1 \le k \le n$ we write $\mathcal{E}_k(D)$ for the set of elementary subdigraphs covering exactly $k$ vertices, and $\mathcal{E}_n(D)$ for the \emph{spanning} elementary subdigraphs. For $\mathsf{E} \in \mathcal{E}_k(D)$ we let $\mathcal{C}(\mathsf{E})$ denote the set of directed cycles composing $\mathsf{E}$ and $c(\mathsf{E}) = | \mathcal{C}(\mathsf{E})|$ their number.

 The gain of a directed walk
$W = v_{i_0} \to v_{i_1} \to \cdots \to v_{i_k}$ is defined as
\[
  \varphi(W) = \prod_{s=0}^{k-1}\varphi(v_{i_s},v_{i_{s+1}}) \in \mathbb{T},
\]
and for $C \in \mathcal{C}(\mathsf{E})$ its gain $\varphi(C)$ is this product taken once around the cycle. The \emph{gain} of an elementary subdigraph $\mathsf{E}$ is the product of the gains of its cycles,
\[
  \varphi(\mathsf{E}) = \prod_{C \in \mathcal{C}(\mathsf{E})}\varphi(C) \in \mathbb{T}.
\]

Let $\ell(C)$ denote the length of a directed cycle $C$. We call
$\Phi = (D,\varphi)$ \emph{cycle balanced} if $\varphi(C) = 1$ for every directed cycle $C$ of $D$, and \emph{cycle antibalanced} if
$\varphi(C) = (-1)^{\ell(C)}$ for every directed cycle $C$ of $D$.
Equivalently, $\Phi$ is antibalanced precisely when $(D,-\varphi)$ is
balanced, since reversing the sign of every arc multiplies the gain of a length-$\ell$ cycle by $(-1)^{\ell}$.

\begin{definition}\label{def:switching}
Given $\sigma\colon V \to \mathbb{T}$, the \emph{switching} of $\Phi$ by $\sigma$ is $\Phi^{\sigma} = (D,\varphi^{\sigma})$, where
$\varphi^{\sigma}(v_i,v_j) = \sigma(v_i)\,\varphi(v_i,v_j)\,\overline{\sigma(v_j)}$. Two $\mathbb{T}$-gain digraphs $\Phi,\Phi'$ on the same digraph $D$ are \emph{switching equivalent}, written $\Phi \sim \Phi'$, if $\Phi' = \Phi^{\sigma}$ for some such $\sigma$.
\end{definition}

Consider a matrix $S = \operatorname{diag}(\sigma(v_1),\dots,\sigma(v_n))$, which is unitary because $|\sigma(v_i)| = 1$, then switching is exactly conjugation by $S$: $A(\Phi^{\sigma}) = S\,A(\Phi)\,S^{-1}$.  A \emph{potential function} for $\varphi$ is a map $\psi\colon V \to \mathbb{T}$ with $\varphi(v_i,v_j) = \overline{\psi(v_i)}\,\psi(v_j)$ for every arc $(v_i,v_j)$; taking $\sigma = \overline{\psi}$ then  the function yields $\varphi^{\sigma} \equiv 1$. So the existence of a potential function means exactly that $\Phi \sim (D,1)$, a condition that will turn out, under strong connectivity, to coincide with balance and with cospectrality with $D$. Here $\Phi$ is \emph{cospectral with $D$} if $A(\Phi)$ and $A(D)$ have the same characteristic polynomial.

The determinant of the gain adjacency matrix admits a combinatorial expansion in terms of spanning elementary subdigraphs of the underlying digraph. This provides a directed analogue of the classical determinant formula for graphs.

\begin{theorem}
  \label{thm:sachs-det}
  Let $\Phi = (D, \varphi)$ be a $\mathbb{T}$-gain digraph on $n$ vertices with gain adjacency matrix $A(\Phi)$. Let $\mathcal{E}_n(D)$ denote the set of all spanning elementary subdigraphs of $D$. For $\mathsf{E} \in \mathcal{E}_n(D)$, 
  %let $c(\mathsf{E})$ denote the number of disjoint directed cycles composing $\mathsf{E}$, and 
  let $\varphi(\mathsf{E}) = \prod_{C \in \mathsf{E}} \varphi(C)$ denote its gain. Then
  \[
    \det(A(\Phi)) = \sum_{\mathsf{E} \in \mathcal{E}_n(D)} (-1)^{n - c(\mathsf{E})} \varphi(\mathsf{E}).
  \]
\end{theorem}

\begin{proof}
 The determinant of $A(\Phi)$ is given by
 \begin{equation}\label{eq1}
\det(A(\Phi)) = \sum_{\pi \in S_n} \operatorname{sgn}(\pi)
\prod_{i=1}^{n} A(\Phi)_{i,\pi(i)}.
\end{equation}
  Every permutation $\pi \in S_n$ decomposes uniquely into a product of disjoint cyclic permutations. This cycle decomposition corresponds bijectively to a spanning elementary subdigraph $\mathsf{E}_\pi \in \mathcal{E}_n(D)$, where a directed arc goes from $i$ to $\pi(i)$ for all $i \in \{1, \dots, n\}$. If any arc $(i, \pi(i))$ is not present in the arc set of $D$, then $A(\Phi)_{i,\pi(i)} = 0$, and the corresponding product vanishes. Thus, the summation can be restricted exclusively to permutations that define valid spanning elementary subdigraphs $\mathsf{E} \in \mathcal{E}_n(D)$.

  For a given permutation $\pi$ associated with the spanning elementary subdigraph $\mathsf{E}$, let $c(\mathsf{E})$ be its number of disjoint cycles. The sign of the permutation is given by the standard relation $\operatorname{sgn}(\pi) = (-1)^{n - c(\mathsf{E})}$. Furthermore, the product of the matrix elements along the cycles of $\pi$ yields the gain of the subdigraph:
  \[
    \prod_{i=1}^n A(\Phi)_{i,\pi(i)} = \prod_{C \in \mathsf{E}} \varphi(C) = \varphi(\mathsf{E}).
  \]
  Substituting these into Equation (\ref{eq1}) yields
  \[
    \det(A(\Phi)) = \sum_{\mathsf{E} \in \mathcal{E}_n(D)} (-1)^{n - c(\mathsf{E})} \varphi(\mathsf{E}),
  \]
  which completes the proof.
\end{proof}

As an immediate consequence of the determinant expansion, we obtain a combinatorial description of the coefficients of the characteristic polynomial.

\begin{corollary}
  \label{cor:sachs-charpoly}
  Let $\Phi = (D, \varphi)$ be a $\mathbb{T}$-gain digraph on $n$ vertices. Let $\mathcal{E}_k(D)$ be the set of all elementary subdigraphs of $D$ covering exactly $k$ vertices. Then the characteristic polynomial of $A(\Phi)$ is
  \[
    p_A(\Phi, \lambda) = \det(\lambda I - A(\Phi)) = \sum_{k=0}^{n} (-1)^k a_k(\Phi)\,\lambda^{n-k}, \] 
  
  where $ a_0(\Phi) = 1$ and 
  $    a_k(\Phi) = \displaystyle\sum_{\mathsf{E} \in \mathcal{E}_k(D)} (-1)^{c(\mathsf{E})}\,\varphi(\mathsf{E}), \quad \text{for } k=1,\ldots,n.
  $
  
\end{corollary}

We now establish a spectral characterization of cycle-balanced $\TT$-gain digraphs. The following theorem shows that cycle balance is completely determined by the spectrum of the gain adjacency matrix.

\begin{theorem}
  \label{thm:main1}
  Let $\Phi = (D, \vp)$ be a $\TT$-gain digraph. Then $\Phi$ is cycle
  balanced if and only if $A(\Phi)$ is cospectral with $A(D)$.
\end{theorem}

\begin{proof}
For the forward implication, it suffices to show that
$a_k(\Phi)=a_k(D)$ for every $k=0,1,\ldots,n$. By
Corollary~\ref{cor:sachs-charpoly}, we have
\[
a_k(\Phi)=\sum_{\mathsf{E}\in\mathcal{E}_k(D)}(-1)^{c(\mathsf{E})}\vp(\mathsf{E}).
\]
Since $\Phi$ is cycle balanced, $\vp(\mathsf{E})=1$ for every
$\mathsf{E}\in\mathcal{E}_k(D)$. Hence $a_k(\Phi)=\sum_{\mathsf{E}\in\mathcal{E}_k(D)}(-1)^{c(\mathsf{E})}=a_k(D)$.
Therefore $A(\Phi)$ and $A(D)$ have the same characteristic polynomial, and so $\Phi$ is cospectral with $D$. Conversely, assume that $A(\Phi)$ and $A(D)$ are cospectral. Then they have the same eigenvalues with multiplicities, and hence  $\tr(A(\Phi)^s) = \tr(A(D)^s)$ for all $ s \ge 1.$ Let $C$ be an arbitrary directed cycle of length $p$. The $(i,i)$-entry of $A(\Phi)^p$ equals the sum of the gains of all directed  walks of length $p$
from $v_i$ to itself. Hence,
\[
\tr(A(\Phi)^p)=\sum_{\substack{W\text{ closed}\\|W|=p}}\varphi(W),
\]
where the sum is over all closed directed walks of length $p$ in $D$.
Similarly, $\tr(A(D)^p)=N_p$, where $N_p$ denotes the number of closed directed walks of length $p$ in $D$. Since $\Phi$ is cospectral with $D$, we have
\[
\sum_{\substack{W\text{ closed}\\|W|=p}}\varphi(W)=N_p.
\]
Each summand has modulus $1$, while the sum equals the number of summands. Hence, by the equality case of the triangle inequality, every closed directed walk of length $p$ has gain $1$. In particular, the walk obtained by traversing $C$ once has gain $1$, so $\varphi(C)=1$. Since $C$ was arbitrary, every directed cycle has gain $1$, and therefore $\Phi$ is cycle balanced.
\end{proof}

Combining the preceding characterization with the notion of switching equivalence, we obtain the following equivalence for strongly connected $\TT$-gain digraphs.

\begin{theorem}
  \label{thm:main2}
  Let $\Phi = (D, \vp)$ be a $\TT$-gain digraph where $D$ is strongly
  connected. Then the following are equivalent:
  \begin{enumerate}[label=\textnormal{(\roman*)}]
    \item $\Phi$ is cycle balanced.
    \item $\Phi \sim (D, \mathbf{1})$, i.e., $\Phi$ has a potential function.
    \item $A(\Phi)$ is cospectral with $A(D)$.
  \end{enumerate}
\end{theorem}
\begin{proof}
The implication \textnormal{(ii)}$\Rightarrow$\textnormal{(iii)} follows immediately, since the existence of a potential function is equivalent to switching equivalence with $(D,\mathbf{1})$. Finally,
\textnormal{(iii)}$\Rightarrow$\textnormal{(i)} is precisely the converse implication of Theorem~\ref{thm:main1}. We prove \textnormal{(i)}$\Rightarrow$\textnormal{(ii)}. Fix a root vertex
$r\in V$. Since $D$ is strongly connected, for each $v\in V$ there exists a directed path $P_{r\to v}$ from $r$ to $v$. Define
\[
\psi(r)=1,\qquad
\psi(v)=\vp(P_{r\to v}).
\]

We first show that $\psi$ is well defined. Let $P$ and $P'$ be two directed paths from $r$ to $v$. By strong connectivity, there exists a directed path $Q$ from $v$ to $r$. Then $P\cup Q$ and $P'\cup Q$ are directed closed walks. Since $\Phi$ is cycle balanced,
\[
\vp(P)\vp(Q)=1=\vp(P')\vp(Q).
\]
As $\vp(Q)\in\TT$ is invertible, it follows that
$\vp(P)=\vp(P')$. Hence $\psi(v)$ is independent of the choice of path.

Now let $(v_i,v_j)\in A(D)$. The path
$P_{r\to v_i}\cup\{(v_i,v_j)\}$ is a directed path from $r$ to $v_j$, so
\[
\psi(v_j)
=\vp(P_{r\to v_i})\,\vp(v_i,v_j)
=\psi(v_i)\vp(v_i,v_j).
\]
Therefore
\[
\vp(v_i,v_j)
=\psi(v_i)^{-1}\psi(v_j)
=\overline{\psi(v_i)}\,\psi(v_j),
\]
showing that $\psi$ is a potential function.
\end{proof}

\begin{example}\label{ex:diamond}
The hypothesis of strong connectivity in Theorem~\ref{thm:main2} cannot be dropped: without it the implication \textnormal{(i)}$\Rightarrow$\textnormal{(ii)} already fails. Let $D$ be the \emph{diamond} on $V=\{u,x,y,w\}$ with arcs $u\to x,\ x\to w,\ u\to y,\ y\to w$ as shown in Figure~\ref{Fig 2}.

\begin{figure}[ht]
\centering
\begin{tikzpicture}[>=stealth,
  every node/.style={circle, draw, minimum size=6mm, inner sep=1pt}]
  \node (u) {$u$};
  \node (x) [above right of=u, node distance=2cm] {$x$};
  \node (y) [below right of=u, node distance=2cm] {$y$};
  \node (w) [below right of=x, node distance=2cm] {$w$};
  \draw[->] (u) to node[draw=none, above left] {$\zeta$} (x);
  \draw[->] (x) to node[draw=none, above right] {$\zeta$} (w);
  \draw[->] (u) to node[draw=none, below left] {$\overline{\zeta}$} (y);
  \draw[->] (y) to node[draw=none, below right] {$\overline{\zeta}$} (w);
\end{tikzpicture}
\caption{A cycle balanced digraph cospectral with $A(D)$ having no potential function.}
\label{Fig 2}
\end{figure}
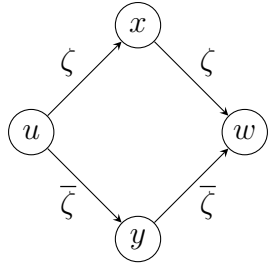
Let $\zeta=e^{\mathrm{i}\pi/3}$ be a primitive sixth root of unity. Consider the gains
\(
\vp(u,x)=\vp(x,w)=\zeta, \vp(u,y)=\vp(y,w)=\overline{\zeta}.
\)
  Since $D$ is acyclic, it is not strongly connected yet it is vacuously cycle balanced. Its gain adjacency matrix is strictly upper triangular with the same all-zero spectrum as $A(D)$, so \textnormal{(iii)} also holds. However, no potential function exists. Without loss of generality let $\psi(u)=1$, since a potential function is determined only up to a common unit scalar. The path $u\to x\to w$ gives $\psi(w)=\vp(u,x)\vp(x,w)=\zeta^{2}=e^{2\pi \mathrm{i}/3}$, whereas the path $u\to y\to w$ gives $\psi(w)=\vp(u,y)\vp(y,w)=\zeta^{-2}=e^{-2\pi \mathrm{i}/3}$, a contradiction since $e^{2\pi \mathrm{i}/3}\neq e^{-2\pi \mathrm{i}/3}$.  Hence \textnormal{(ii)} fails while \textnormal{(i)} and \textnormal{(iii)} hold. 

\end{example}

\section{Bounds on Spectral Radius} \label{sec 3}

The spectral radius which is the maximum modulus of its eigenvalues is one of the most important spectral invariants associated with the adjacency matrix of a digraph. In this section, we establish general bounds for the spectral radius of a $\TT$-gain digraph denoted by $\rho(\Phi)$ in terms of structural parameters of its underlying digraph. We first compare the spectral radius with that of the underlying digraph and characterize the equality case, and then derive further upper and lower bounds. To prove these results, we first recall several known facts from matrix theory and spectral graph theory that will be used throughout this section.

\begin{theorem}[ {\cite[p.~534]{HornJohnson}}] \label{a}
Let $M \in M_n(\mathbb{R})$ be a nonnegative irreducible matrix with $n\ge2$. Then the following properties hold:

\begin{enumerate}
    \item The spectral radius $\rho(M)$ is strictly positive.
    \item The eigenvalue $\rho(M)$ has algebraic multiplicity one.
    \item There exists a unique positive right eigenvector $x$ (up to normalization) satisfying $Mx=\rho(M)x$.
    \item There exists a unique positive left eigenvector $y$ (up to normalization) satisfying $y^{T}M=\rho(M)y^{T}$.
\end{enumerate}
\end{theorem}

Theorem \ref{a} is the Perron-Frobenius theorem for nonnegative irreducible matrices, which provides fundamental properties of the spectral radius and its associated positive eigenvectors. Since the adjacency matrix of a digraph is a nonnegative matrix, it is natural to relate its irreducibility to the connectivity structure of the underlying digraph. The following result establishes this correspondence by characterizing irreducibility in terms of strong connectivity, allowing us to apply Theorem \ref{a} to the adjacency matrix of a strongly connected digraph.

\begin{theorem}[{\cite[p.~55]{BrualdiRyser1991}}] \label{b}
Let $D=(V,\mathcal{A})$ be a digraph with adjacency matrix $A$. Then $A$ is irreducible if and only if $D$ is strongly connected.
\end{theorem}

We now introduce the following generalization of cycle balance.

\begin{definition}
Let $\Phi=(D,\varphi)$ be a $\TT$-gain digraph and let $\mu\in\TT$. We say that $\Phi$ is \emph{$\mu$-balanced} if $\varphi(C)=\mu^{|C|}$ for every directed cycle $C$ of $D$ where $|C|$ is the length of the directed cycle.
\end{definition}

The next result characterizes $\mu$-balanced gain digraphs in terms of switching equivalence with a constant gain digraph.

\begin{lemma}\label{lem:mubalance}
Let $\Phi=(D,\varphi)$ be a $\TT$-gain digraph with strongly connected underlying digraph $D$, and let $\mu\in\TT$. Then $\Phi$ is
$\mu$-balanced if and only if it is switching equivalent to the
constant gain digraph $(D,\mu\cdot\mathbf{1})$, where every arc of $D$ is assigned the same gain $\mu\in\TT$.
\end{lemma}
\begin{proof}
Define $\psi(u,v)=\mu^{-1}\varphi(u,v)$ for each arc $(u,v)$ and let
$\Psi=(D,\psi)$. Then, for every directed cycle $C$,
\[
\psi(C)=\mu^{-|C|}\varphi(C).
\]
Hence $\Phi$ is $\mu$-balanced if and only if $\Psi$ is balanced. Since
$D$ is strongly connected, Theorem~\ref{thm:main2} implies that
$\Psi$ is balanced if and only if $\Psi\sim(D,\mathbf1)$. If $\Psi^\sigma=(D,\mathbf1)$, then $1=\sigma(u)^{-1}\mu^{-1}\varphi(u,v)\sigma(v)$, for every arc $(u,v)$, so $\varphi^\sigma(u,v)=\mu$. Thus $\Phi^\sigma=(D,\mu\cdot\mathbf1)$, proving $\Phi\sim(D,\mu\cdot\mathbf1)$.

Conversely, if $\Phi\sim(D,\mu\cdot\mathbf1)$, then switching preserves cycle gains, and every directed cycle of $(D,\mu\cdot\mathbf1)$ has gain $\mu^{|C|}$. Hence $\Phi$ is $\mu$-balanced.
\end{proof}

We now state the following results due to Horn and Johnson \cite{HornJohnson}.

\begin{theorem}[{\cite[P.~520]{HornJohnson}}]
\label{thm:horn}
Let $M,N\in M_n$, and suppose that $N$ is nonnegative. If
\(
|M|\le N,
\)
then
\(
\rho(M)\le \rho(|M|)\le \rho(N).
\)
\end{theorem}

\begin{theorem}[{\cite[p.~533]{HornJohnson}}]\label{thm:wielandt}
Let $M, N \in M_n$. Suppose that $N$ is nonnegative and irreducible, and
\(
N \ge |M|.
\)
Let $\lambda=e^{\mathrm{i}\theta}\rho(M)$ be a given maximum-modulus eigenvalue of $M$. If
\(
\rho(N)=\rho(M),
\)
then there is a diagonal unitary matrix $S\in M_n$ such that
\(
M=e^{\mathrm{i}\theta}SNS^{-1}.
\)
\end{theorem}

The following theorem is a consequence of Theorem~\ref{thm:horn}, Theorem\ref{thm:wielandt} and Lemma \ref{lem:mubalance}.

\begin{theorem}\label{thm:rho-comparison}
Let $\Phi=(D,\varphi)$ be a $\TT$-gain digraph whose underlying digraph $D$ is strongly connected. Then $\rho(\Phi)\le\rho(D)$, with equality if and only if $\Phi$ is $\mu$-balanced for some $\mu\in\TT$.
\end{theorem}

\begin{proof}
Since $D$ is strongly connected, $A(D)$ is nonnegative and irreducible
by Theorem~\ref{b}. Every nonzero entry of $A(\Phi)$ has modulus one,
so $|A(\Phi)|=A(D)$, and Theorem~\ref{thm:horn} gives
$\rho(\Phi)\le\rho(D)$.

Suppose $\rho(\Phi)=\rho(D)$, and let $\lambda\in\spec(A(\Phi))$ with
$|\lambda|=\rho(D)$. Write $\lambda=\mu\rho(D)$ with
$\mu=\lambda/|\lambda|\in\TT$. By
Theorem~\ref{thm:wielandt} there is a diagonal $S\in M_n(\CC)$ with
$|S|=I$ such that $A(\Phi)=\mu\,S\,A(D)\,S^{-1}$. Writing
$S=\operatorname{diag}(\sigma(v_1),\dots,\sigma(v_n))$ with
$\sigma\colon V\to\TT$, this reads
$\varphi(v_i,v_j)=\mu\,\sigma(v_i)\,\overline{\sigma(v_j)}$ for each
arc $(v_i,v_j)$, that is, $\Phi^{\sigma^{-1}}=(D,\mu\cdot\mathbf 1)$.
Hence $\Phi\sim(D,\mu\cdot\mathbf 1)$, and $\Phi$ is $\mu$-balanced by
Lemma~\ref{lem:mubalance}.

Conversely, if $\Phi$ is $\mu$-balanced then
$\Phi\sim(D,\mu\cdot\mathbf 1)$ by Lemma~\ref{lem:mubalance}, so
$A(\Phi)$ is similar to $\mu A(D)$ and
$\spec(A(\Phi))=\mu\,\spec(A(D))$. As $|\mu|=1$, $\rho(\Phi)=\rho(D)$.
\end{proof}

If some pair $\{v_i,v_j\}$ forms a digon with conjugate gains
$\vp(v_j,v_i)=\overline{\vp(v_i,v_j)}$, then this $2$-cycle has gain
$\vp(v_i,v_j)\vp(v_j,v_i)=|\vp(v_i,v_j)|^2=1$, and $\mu$-balance applied to it gives $\mu^{2}=1$, hence $\mu\in\{+1,-1\}$. In particular, once the gains are reciprocated conjugately, only balance and antibalance can attain $\rho(\Phi)=\rho(D)$. This generalises, in the undirected case, the result of Mehatari et al.~\cite{Mehatari2022}, who proved that the spectral radius of a complex unit gain graph equals that of its underlying graph precisely when the gain graph is balanced or antibalanced.

\begin{theorem}\label{them:upp}
  \label{thm:geom-bound}
  For any $\TT$-gain digraph $\Phi=(D,\vp)$, $\rho(\Phi) \leq \Delta^+$.
\end{theorem}

  \begin{proof}
  Let $\lambda$ be an eigenvalue with eigenvector $x \neq 0$.
  Let $v_p$ be the vertex such that $|x_p| = \max_i|x_i|$.
  From $A(\Phi)x = \lambda x$:
  \[
    |\lambda|\,|x_p|
    = \left|\sum_{j:(v_p,v_j)\in \mathcal{A}}\vp(v_p,v_j)\,x_j\right|
    \leq \sum_{(v_p,v_j)\in \mathcal{A}}|x_j|
    \leq \deg^+(v_p)|x_p|
    \leq \Delta^+|x_p|,
  \]
  so $|\lambda|\leq\Delta^+$. 
  \end{proof}
  
As an immediate consequence of Theorem \ref{them:upp} and Theorem \ref{thm:rho-comparison} we get the following corollary.

  \begin{corollary}
\label{cor:geom-bound-equality}
Let $\Phi=(D,\varphi)$ be a $\TT$-gain digraph whose underlying digraph $D$ is strongly connected. Then, $\rho(\Phi)=\Delta^+$ if and only if $D$ is $\Delta^+$-out-regular and $\Phi$ is $\mu$-balanced for some $\mu\in\TT$.
\end{corollary}

%\begin{proof}
%The result follows immediately from Theorem~\ref{thm:rho-comparison} together with the fact that, for a strongly connected digraph $D$, $\rho(D)=\Delta^+$ if and only if $D$ is $\Delta^+$-out-regular.
%\end{proof}

We conclude this section by establishing a lower bound for the spectral radius in terms of the number of digons when opposite arcs carry conjugate gains.

\begin{theorem}\label{thm:lower-digon}
Let $\Phi=(D,\vp)$ be a $\TT$-gain digraph in which every digon carries conjugate gains. Then
\[
  \rho(\Phi)\ \ge\ \sqrt{\frac{2q}{n}},
\]
 where $q$ is the number of digons of $D$.
\end{theorem}
\begin{proof}
Let $\lambda_1, \lambda_2, \ldots,\lambda_n$ be the eigenvalues of $A(\Phi)$. Then
\[
\rho(\Phi)^2
=
\max_i |\lambda_i|^2
\geq
\frac1n\sum_{i=1}^n |\lambda_i|^2
\geq
\frac1n\left|\sum_{i=1}^n \lambda_i^2\right|.
\]
Since
$\sum_{i=1}^n \lambda_i^2
=
\operatorname{tr}(A(\Phi)^2)$,
and each digon contributes $1$ to two diagonal entries of $A(\Phi)^2$, we have $\operatorname{tr}(A(\Phi)^2)=2q.$
Therefore,
$\rho(\Phi)^2
\geq
\frac{2q}{n}.$ Hence, the result follows.
\end{proof}

\section{Cospectrality of Unicyclic Digraphs} \label{sec4}

Unicyclic digraphs are the simplest structures beyond directed forests, and they isolate cleanly the role played by a single cycle gain. We show that for a $\TT$-gain digraph carrying one directed cycle the entire spectrum is determined by only three data the number of vertices $n$, the cycle length $\ell$, and the single cycle gain $g=\varphi(C_\ell)$ while the gains distributed along the tree arcs are spectrally invisible.

We begin with the spectral characterization of a directed cycle. The following theorem shows that its spectrum is determined entirely by the length of the cycle and its total gain, with the individual arc gains otherwise having no spectral effect.

\begin{theorem}\label{thm:cycle-spectrum}
Let $\Phi=(C_n,\varphi)$ be a $\mathbb{T}$-gain digraph on the directed cycle $v_1\to v_2\to\cdots\to v_n\to v_1$. If $g=\varphi(C_n)=\prod_{i=1}^{n}\varphi(v_i,v_{i+1})$, then $\operatorname{spec}(A(\Phi)) = \{\lambda\in\mathbb{C}:\lambda^n=g\}$ and $\rho(\Phi)=1$.
\end{theorem}

\begin{proof}
Let $S=\operatorname{diag}(s_1,s_2, \ldots,s_n)$ 
be the unitary matrix defined by $s_1=1$ and $s_k=\prod_{j=1}^{k-1}\varphi(v_j,v_{j+1})$ for $k\ge 2$. Under the switching similarity transformation $B=SA(\Phi)S^{-1}$, the entries become $B_{k,k+1} = s_k\varphi(v_k,v_{k+1})s_{k+1}^{-1} = 1$ for $1\le k < n$, and $B_{n1} = s_n\varphi(v_n,v_1) = g$. Thus, 
\[
B=
\begin{pmatrix}
0&1&0&\cdots&0\\
0&0&1&\cdots&0\\
\vdots&&&\ddots&\vdots\\
0&0&0&\cdots&1\\
g&0&0&\cdots&0
\end{pmatrix}.
\] 
Hence we get the characteristic polynomial as $\det(\lambda I-B)=\lambda^n-g$. Since $A(\Phi)$ and $B$ are similar, the eigenvalues of $A(\Phi)$ are the roots of $\lambda^n=g$. Because $|g|=1$, every eigenvalue satisfies $|\lambda|=1$, yielding $\rho(\Phi)=1$.This completes the proof.
\end{proof}

The Theorem \ref{thm:cycle-spectrum} gives a characterization of cospectrality for $\TT$-gain directed cycles in terms of their total cycle gains.

\begin{corollary}
\label{cor:cycle-cospectral}
Let $\Phi_1=(C_n,\varphi_1)$ and $\Phi_2=(C_n,\varphi_2)$ be two
$\TT$-gain digraphs. Then $\Phi_1$ and $\Phi_2$ are cospectral if and
only if
\[
\varphi_1(C_n)=\varphi_2(C_n).
\]
\end{corollary}

\begin{definition}
A digraph $D$ on $n$ vertices is called \emph{unicyclic} if it is weakly connected and its underlying undirected graph is unicyclic. 

Throughout this section, we consider unicyclic digraphs whose unique cycle is directed. Thus every remaining arc belongs to a directed tree whose arcs are oriented toward the cycle.
\end{definition}

The spectral structure of a unicyclic digraph is governed by its unique directed cycle, while the attached directed trees contribute only zero eigenvalues.

\begin{theorem}
Let $\Phi=(D,\varphi)$ be a $\TT$-gain unicyclic digraph on $n$ vertices whose unique directed cycle is $C_\ell$, and let
$g=\varphi(C_\ell)\in\TT$.
Then
\[
\operatorname{spec}(A(\Phi))
=
\{0^{\,n-\ell}\}
\cup
\{\lambda\in\mathbb C:\lambda^\ell=g\},
\]
where $0^{\,n-\ell}$ denotes the eigenvalue $0$ with multiplicity
$n-\ell$. Consequently, $\rho(\Phi)=1.$
\end{theorem}

\begin{proof}
Order the vertices which are not lying on the directed cycle first (from the tips toward the cycle), followed by the vertices of the directed cycle.
With this ordering,
\[
A(\Phi)=
\begin{pmatrix}
N&R\\
0&C(\Phi)
\end{pmatrix},
\]
where $N$ is a strictly upper triangular matrix corresponding to the arcs outside the directed circle, $R$ contains the arcs joining the tails to the cycle, and $C(\Phi)$ is the gain adjacency matrix of the directed cycle
$C_\ell$.

Since $N$ is strictly upper triangular, it is nilpotent. Hence all
eigenvalues of $N$ are equal to $0$, with algebraic multiplicity
$n-\ell$. Since $A(\Phi)$ is block upper triangular, its spectrum is the union of the spectra of its diagonal blocks. Therefore,
$\operatorname{spec}(A(\Phi))
=
\operatorname{spec}(N)
\cup
\operatorname{spec}(C(\Phi))$.
By Theorem~\ref{thm:cycle-spectrum},
$\operatorname{spec}(C(\Phi))
=
\{\lambda\in\mathbb C:\lambda^\ell=g\}.$
Thus
\[
\operatorname{spec}(A(\Phi))
=
\{0^{\,n-\ell}\}
\cup
\{\lambda\in\mathbb C:\lambda^\ell=g\}.
\]
Finally, since $g\in\TT$, every solution of $\lambda^\ell=g$ has modulus one. Hence every nonzero eigenvalue of $A(\Phi)$ has modulus one, while the remaining eigenvalues are zero. Therefore,
$\rho(\Phi)=1.$
\end{proof}

As a consequence, the spectrum of a $\TT$-gain unicyclic digraph is completely determined by the number of vertices, the length of its unique directed cycle, and the total gain of that cycle.

\begin{corollary}
\label{cor:unicyclic-cospectral}
Let $\Phi_1=(D_1,\varphi_1)$ and $\Phi_2=(D_2,\varphi_2)$ be two
$\TT$-gain unicyclic digraphs with unique directed cycles of lengths
$\ell_1$ and $\ell_2$, respectively. Then $\Phi_1$ and $\Phi_2$ are
cospectral if and only if
\[
|V(D_1)|=|V(D_2)|,\qquad
\ell_1=\ell_2,
\qquad\text{and}\qquad
\varphi_1(C_{\ell_1})=\varphi_2(C_{\ell_2}).
\]
\end{corollary}

\section{Laplacian matrix of $\TT$-gain Digraph} \label{sec5}

The Laplacian and signless Laplacian matrices of a $\TT$-gain digraph
$\Phi=(D,\varphi)$ are defined by
\[
L(\Phi)=D^{+}-A(\Phi), \qquad
Q(\Phi)=D^{+}+A(\Phi),
\]
where $D^{+}=\operatorname{diag}(\deg^{+}(v_1),\dots,\deg^{+}(v_n))$
is the diagonal matrix of out-degrees. Since the gains
$\varphi(v_i,v_j)$ and $\varphi(v_j,v_i)$ are unrelated, both
$L(\Phi)$ and $Q(\Phi)$ are, in general, non-normal, and their spectra lie in $\mathbb{C}$. Consequently, the variational and majorization techniques available for the Hermitian Laplacian of a gain graph are no longer applicable. We discuss the combinatorial interpretation of determinant and characteristic polynomial for both $L(\Phi)$ and $Q(\Phi)$. Moreover, we discuss the cycle balance criteria in terms of the eigenvalues for both the matrices and  the spectra of both matrices are shown to lie in the closed right half-plane.

\begin{theorem}\label{thm:coates}
Let $\Phi=(D,\varphi)$ be a $\mathbb{T}$-gain digraph on $n$ vertices. Then
\[
\det L(\Phi)
=
\sum_{\mathsf{E}\in\mathcal{E}(D)}
(-1)^{c(\mathsf{E})}
\varphi(\mathsf{E})
\prod_{v\notin V(\mathsf{E})}
\deg^{+}(v),
\]
and
\[
\det Q(\Phi)
=
\sum_{\mathsf{E}\in\mathcal{E}(D)}
(-1)^{|V(\mathsf{E})|+c(\mathsf{E})}
\varphi(\mathsf{E})
\prod_{v\notin V(\mathsf{E})}
\deg^{+}(v),
\]
where the summation is over all elementary subdigraphs of $D$.
\end{theorem}

\begin{proof}
The proof is similar to that of Theorem \ref{thm:sachs-det}. Thus, each nonzero permutation corresponds uniquely to an elementary subdigraph of $D$, fixed points contribute the diagonal degree terms, directed cycles contribute their gain products, and the resulting signs simplify to $(-1)^{c(\mathsf{E})}$ for $L(\Phi)$ and $(-1)^{|V(\mathsf{E})|+c(\mathsf{E})}$ for $Q(\Phi)$.
\end{proof}

The determinant expansions above extend naturally to the characteristic polynomials of the Laplacian and signless Laplacian matrices.

\begin{corollary}\label{cor:coates}
Let $\Phi=(D,\varphi)$ be a $\mathbb{T}$-gain digraph on $n$ vertices. Then
\[
p_L^{}(\Phi,\lambda)
=
\det(\lambda I-L(\Phi))
=
\sum_{\mathsf{E}\in\mathcal{E}(D)}
(-1)^{|V(\mathsf{E})|-c(\mathsf{E})}
\varphi(\mathsf{E})
\prod_{v\notin V(\mathsf{E})}
\bigl(\lambda-\deg^{+}(v)\bigr),
\]
and
\[
p_Q^{}(\Phi,\lambda)
=
\det(\lambda I-Q(\Phi))
=
\sum_{\mathsf{E}\in\mathcal{E}(D)}
(-1)^{c(\mathsf{E})}
\varphi(\mathsf{E})
\prod_{v\notin V(\mathsf{E})}
\bigl(\lambda-\deg^{+}(v)\bigr).
\]
\end{corollary}

We now characterize cycle balance in terms of the Laplacian spectrum. The following theorem is the Laplacian analogue of Theorem~\ref{thm:main2}.

\begin{theorem}\label{thm:sc-balance}
Let $\Phi=(D,\varphi)$ be a $\mathbb{T}$-gain digraph whose underlying digraph
$D$ is strongly connected. Then the following are equivalent.
\begin{enumerate}[label=\textup{(\roman*)}]
\item $\Phi$ is cycle balanced.
\item $\Phi\sim(D,\mathbf1)$, i.e., $\Phi$ has a potential function.
\item $L(\Phi)$ is cospectral with $L(D)$.
\item $0\in\spec L(\Phi)$.
\end{enumerate}
\end{theorem}

\begin{proof}
The equivalence of {\rm(i)}, {\rm(ii)} and {\rm(iii)} follows directly from Theorem~\ref{thm:main2}. Since $0$ is an eigenvalue of $L(D)$, therefore {\rm(iii)} $\Rightarrow$ {\rm(iv)}. Hence, it suffices to prove
{\rm(iv) $\Rightarrow$ (ii)}.

If $0\in\spec L(\Phi)$, then there exists a non-zero vector $ y\in\ker L(\Phi)$ with
\[
\deg^{+}(v_i)y_i=\sum_{(v_i,v_j)\in \mathcal{A}}\varphi(v_i,v_j)y_j,
\qquad i=1,2, \dots,n.
\]
Let $m=\max_i|y_i|$ and $U=\{v_i:|y_i|=m\}$. If $v_i\in U$, then
\begin{equation}\label{eq2}
\deg^{+}(v_i)m
=\Big|\sum_{(v_i,v_j)\in \mathcal{A}}\varphi(v_i,v_j)y_j\Big|
\le\sum_{(v_i,v_j)\in \mathcal{A}}|y_j|
\le\deg^{+}(v_i)m,
\end{equation}
 From \eqref{eq2}, we have $|y_j|=m$ for every out-neighbour $v_j$ of $v_i$, and thus $U$ is closed under taking out-neighbours. Since $D$ is strongly connected, $U=V$, i.e.,
$|y_i|=m$ for all $1\leq i\leq n$. From \eqref{eq2}, it is easy to verify that
\[
\varphi(v_i,v_j)y_j=y_i
\qquad \forall \ (v_i,v_j)\in \mathcal{A}.
\]
Consider $\sigma(v_i)=\overline{y_i}/m$. Then it follows that, $|\sigma(v_i)|=1$, and
\(
\varphi(v_i,v_j)\overline{\sigma(v_j)}
=\overline{\sigma(v_i)},
\)
which gives us
\(
\varphi(v_i,v_j)
=\overline{\sigma(v_i)}\,\sigma(v_j).
\)
Thus $\sigma$ is a potential function, and hence
$\Phi\sim(D,\mathbf1)$.
\end{proof}

The hypothesis of strong connectivity in Theorem~\ref{thm:sc-balance} is essential. The following example shows a $\TT$-gain digraph that is not strongly connected and not cycle balanced, yet still has $0$ as a Laplacian eigenvalue.
 
\begin{example}\label{ex:weak-zero}
Consider the $\mathbb{T}$-gain digraph $\Phi=(D,\varphi)$ given in
Figure~\ref{fig:weak}. The underlying digraph is weakly connected but not strongly connected. The gain of the lower digon is $1$, whereas the gain of the upper digon is $-1$. Hence $\Phi$ is not cycle balanced.

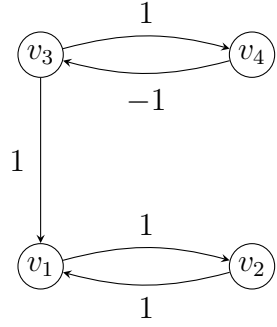
\begin{figure}[ht]
\centering
\begin{tikzpicture}[>=stealth, node distance=2.8cm,
  every node/.style={circle, draw, minimum size=6mm, inner sep=1pt}]

\node (v3) {$v_3$};
\node (v4) [right of=v3] {$v_4$};
\node (v1) [below of=v3] {$v_1$};
\node (v2) [right of=v1] {$v_2$};

% Top digon
\draw[->] (v3) to[bend left=15] node[draw=none, above] {$1$} (v4);
\draw[->] (v4) to[bend left=15] node[draw=none, below] {$-1$} (v3);

% Bottom digon
\draw[->] (v1) to[bend left=15] node[draw=none, above] {$1$} (v2);
\draw[->] (v2) to[bend left=15] node[draw=none, below] {$1$} (v1);

% Connecting arc
\draw[->] (v3) to node[draw=none, left] {$1$} (v1);

\end{tikzpicture}
\caption{A $\mathbb{T}$-gain digraph which is not strongly connected with $0$ eigenvalue.}
\label{fig:weak}
\end{figure}

It is easy to check that $0\in\spec L(\Phi)$ although $\Phi$ is not cycle balanced. This shows that the strong connectivity assumption in Theorem~\ref{thm:sc-balance} is essential.

% We have,

% \[
% L(\Phi)=
% \begin{pmatrix}
% 1 & -1 & 0 & 0\\
% -1 & 1 & 0 & 0\\
% -1 & 0 & 2 & -1\\
% 0 & 0 & 1 & 1
% \end{pmatrix}.
% \]
% Since
% \[
% L(\Phi)=
% \begin{pmatrix}
% B_1 & 0\\
% * & B_2
% \end{pmatrix},
% \]
% where
% \[
% B_1=
% \begin{pmatrix}
% 1&-1\\
% -1&1
% \end{pmatrix},
% \qquad
% B_2=
% \begin{pmatrix}
% 2&-1\\
% 1&1
% \end{pmatrix},
% \]
% we have $\det L(\Phi)=\det(B_1)\det(B_2)=0.$ Therefore, 

\end{example}

An analogous characterization for the signless Laplacian follows by replacing the gains with their negatives.

\begin{corollary}\label{cor:signless-balance}
Let $\Phi=(D,\varphi)$ be a $\mathbb{T}$-gain digraph whose underlying digraph $D$ is strongly connected. Then $0\in\spec Q(\Phi)$ if and only if $\Phi \text{ is cycle antibalanced}.$ Moreover, when these conditions hold, $Q(\Phi)$ is cospectral with $L(D)$.
\end{corollary}

\begin{proof}
Since
\[
Q(\Phi)=D^{+}+A(\Phi)=D^{+}-A(-\Phi)=L(-\Phi),
\]
we have
\[
0\in\spec Q(\Phi)
\iff
0\in\spec L(-\Phi).
\]
By Theorem~\ref{thm:sc-balance}, this is equivalent to $-\Phi$ being cycle balanced, that is, to $\Phi$ being cycle antibalanced. Furthermore, if $\Phi$ is cycle antibalanced then $Q(\Phi)$ is cospectral with $L(D)$.
\end{proof}

We recall the following form of the Ger\v{s}gorin circle theorem, which will be used to locate the spectrum of the Laplacian matrix.

\begin{theorem} {\cite[p.~388]{HornJohnson}}\label{thm:gersgorin}
Let $M=[m_{ij}]\in\mathbb{C}^{n\times n}$, and for each $i$ let $R_i=\sum_{j\neq i}|m_{ij}|$ denote its $i$-th
deleted absolute row sum. Then every eigenvalue of $M$ lies in 
\(
\bigl\{\,z\in \mathbb{C}:\ |z-m_{ii}|\le R_i\,\bigr\} 
\) for some $i$.
\end{theorem}

Applying the Ger\v{s}gorin theorem to the two Laplacian matrices yields the following spectral localization result.
 
\begin{theorem}\label{thm:LQ-rhp}
Let $\Phi = (D, \varphi)$ be a $\mathbb{T}$-gain digraph. Then every eigenvalue of $L(\Phi)$ and every eigenvalue of $Q(\Phi)$ has nonnegative real part.
\end{theorem}

\begin{proof}
For $L(\Phi) = D^{+} - A(\Phi)$ and $Q(\Phi) = D^{+} + A(\Phi)$ the diagonal entries are $ \deg^{+}(v_i)$, and since every
gain has modulus one, the $i$th deleted absolute row sum for both the matrices is given by
\[
  R_i = \sum_{(v_i,v_j)\in \mathcal{A}} \bigl| \varphi(v_i,v_j) \bigr|
      = \sum_{(v_i,v_j)\in \mathcal{A}} 1 = \deg^{+}(v_i).
\]
From Theorem~\ref{thm:gersgorin}, each eigenvalue of $L(\Phi)$ and $Q(\Phi)$ lies in 
\(
  \bigl\{\, z \in \mathbb{C} : |z - \deg^{+}(v_i)| \le \deg^{+}(v_i) \,\bigr\},
\)
a disc centred on real axis at the nonnegative real number $\deg^{+}(v_i)$ with radius $\deg^{+}(v_i)$. Every such disc is contained in the closed right half-plane and hence every eigenvalue has nonnegative real part.
\end{proof}

Since each eigenvalue of $L(\Phi)$ and $Q(\Phi)$ lies in 
\(
  \bigl\{\, z \in \mathbb{C} : |z - \deg^{+}(v_i)| \le \deg^{+}(v_i) \,\bigr\},
\) it follows that if the underlying digraph $D$ is strongly connected, 
then from Theorem \ref{thm:sc-balance} the only possible eigenvalue of $L(\Phi)$ on the imaginary axis is $0$, which occurs if and only if $\Phi$ is cycle balanced. Likewise, by Corollary \ref{cor:signless-balance} the only possible eigenvalue of $Q(\Phi)$ on the imaginary axis is $0$, occurring if and only if $\Phi$ is antibalanced. Consequently, $L(\Phi)$ is positive stable (real part of its eigenvalues is strictly positive), whenever $\Phi$ is not cycle balanced, and $Q(\Phi)$ is positive stable whenever $\Phi$ is not antibalanced.

\section{Conclusion and Future Scope} \label{sec6}

In this paper, we studied the balance and spectral properties of non-Hermitian matrices associated with $\TT$-gain digraphs. We established bounds on the spectral radius of the adjacency matrix and characterized the cases of equality. We also determined the spectra of unicyclic digraphs.

Most of our characterizations 
%of balance and the equality case 
for the spectral radius require the underlying digraph to be strongly connected, it would therefore be interesting to identify classes of weakly connected digraphs for which analogous results remain valid. We have also restricted our attention to digraphs containing a single directed cycle. A natural continuation is the study of general digraphs whose underlying graph contains more than one cycles.
%that are not directed cycles.
Other directions include determining when the spectrum uniquely determines a gain digraph and minimizing the spectral radius over all gain assignments on a fixed digraph. 

\section{Acknowledgment}
The first author expresses gratitude for the Junior Research Fellowship provided by BITS Pilani. The third author also extends thanks for the support received from NFSG N4/24/1010, granted by BITS Pilani.

\bibliographystyle{plain}
\bibliography{references}

\end{document}